\documentclass[12pt]{amsart}
\usepackage{verbatim}
\usepackage{amssymb, amsmath}
\usepackage{graphicx}
\usepackage{caption}
\usepackage{subcaption}

\usepackage{color}

\begin{document}
\newtheorem{thm}{Theorem}[section]
\newtheorem*{thm*}{Theorem}
\newtheorem{lem}[thm]{Lemma}
\newtheorem{prop}[thm]{Proposition}
\newtheorem{cor}[thm]{Corollary}
\newtheorem*{conj}{Conjecture}
\newtheorem{proj}[thm]{Project}
\newtheorem{question}[thm]{Question}
\newtheorem{rem}{Remark}[section]

\theoremstyle{definition}
\newtheorem*{defn}{Definition}
\newtheorem*{remark}{Remark}
\newtheorem{exercise}{Exercise}
\newtheorem*{exercise*}{Exercise}

\numberwithin{equation}{section}

\newcommand{\rad}{\operatorname{rad}}

\newcommand{\Z}{{\mathbb Z}} 
\newcommand{\Q}{{\mathbb Q}}
\newcommand{\R}{{\mathbb R}}
\newcommand{\C}{{\mathbb C}}
\newcommand{\N}{{\mathbb N}}
\newcommand{\FF}{{\mathbb F}}
\newcommand{\fq}{\mathbb{F}_q}
\newcommand{\rmk}[1]{\footnote{{\bf Comment:} #1}}

\renewcommand{\mod}{\;\operatorname{mod}}
\newcommand{\ord}{\operatorname{ord}}
\newcommand{\TT}{\mathbb{T}}
\renewcommand{\i}{{\mathrm{i}}}
\renewcommand{\d}{{\mathrm{d}}}
\renewcommand{\^}{\widehat}
\newcommand{\HH}{\mathbb H}
\newcommand{\Vol}{\operatorname{vol}}
\newcommand{\area}{\operatorname{area}}
\newcommand{\tr}{\operatorname{tr}}
\newcommand{\norm}{\mathcal N} 
\newcommand{\intinf}{\int_{-\infty}^\infty}
\newcommand{\ave}[1]{\left\langle#1\right\rangle} 
\newcommand{\Var}{\operatorname{Var}}
\newcommand{\Prob}{\operatorname{Prob}}
\newcommand{\sym}{\operatorname{Sym}}
\newcommand{\disc}{\operatorname{disc}}
\newcommand{\CA}{{\mathcal C}_A}
\newcommand{\cond}{\operatorname{cond}} 
\newcommand{\lcm}{\operatorname{lcm}}
\newcommand{\Kl}{\operatorname{Kl}} 
\newcommand{\leg}[2]{\left( \frac{#1}{#2} \right)}  
\newcommand{\Li}{\operatorname{Li}}

\newcommand{\sumstar}{\sideset \and^{*} \to \sum}

\newcommand{\LL}{\mathcal L} 
\newcommand{\sumf}{\sum^\flat}
\newcommand{\Hgev}{\mathcal H_{2g+2,q}}
\newcommand{\USp}{\operatorname{USp}}
\newcommand{\conv}{*}
\newcommand{\dist} {\operatorname{dist}}
\newcommand{\CF}{c_0} 
\newcommand{\kerp}{\mathcal K}

\newcommand{\Cov}{\operatorname{cov}}
\newcommand{\Sym}{\operatorname{Sym}}

\newcommand{\Ht}{\operatorname{Ht}}

\newcommand{\E}{\operatorname{\mathbb E}} 
\newcommand{\sign}{\operatorname{sign}} 
\newcommand{\meas}{\operatorname{meas}} 
\newcommand{\length}{\operatorname{length}} 

\newcommand{\divid}{d} 

\newcommand{\GL}{\operatorname{GL}}
\newcommand{\SL}{\operatorname{SL}}
\newcommand{\re}{\operatorname{Re}}
\newcommand{\im}{\operatorname{Im}}
\newcommand{\res}{\operatorname{Res}}
 \newcommand{\eigen}{\Lambda} 
\newcommand{\tens}{\mathbf t} 
\newcommand{\diam}{\operatorname{diam}}
\newcommand{\fixme}[1]{\footnote{Fixme: #1}}
 \newcommand{\EWp}{\mathbb E^{\rm WP}} 
\newcommand{\orb}{\operatorname{Orb}}
\newcommand{\supp}{\operatorname{Supp}}
\newcommand{\mmfactor }{\textcolor{red}{c_{\rm Mir}}}
\newcommand{\Mg}{\mathcal M_g} 
\newcommand{\MCG}{\operatorname{Mod}} 
\newcommand{\Diff}{\operatorname{Diff}} 
\newcommand{\If}{I_f(L,\tau)}
\newcommand{\SigGOE}{\Sigma^2_{\rm GOE}}

\newcommand{\Nc}{\mathcal{N}}  
\newcommand{\Rpos}{\R_{>0}}
\newcommand{\Rnneg}{\R_{\geq 0}}
\newcommand{\vlim}{ \overset{v}{\rightarrow}}
\newcommand{\dlim}{ \overset{d}{\rightarrow}}
\newcommand{\Pois}{\operatorname{Pois}}

\title[The CLT on random surfaces of large genus]
{On the Central Limit Theorem for linear eigenvalue statistics on random surfaces of large genus}
\author{Ze\'ev Rudnick  and Igor Wigman}
\address{School of Mathematical Sciences, Tel Aviv University, Tel Aviv 69978, Israel} 
\email{rudnick@tauex.tau.ac.il}
\address{Department of Mathematics, King's College London, UK}
\email{igor.wigman@kcl.ac.uk}

\thanks{ This research was supported by the European Research Council (ERC) under the European Union's Horizon 2020 research and innovation programme (grant agreement No. 786758) and by the Israel Science Foundation (grant No. 1881/20).  }

\begin{abstract}  
We study the fluctuations of smooth linear statistics of Laplace eigenvalues of compact hyperbolic surfaces lying in short energy windows, when averaged over the moduli space of surfaces of a given genus.
The average is taken with respect to  the Weil-Petersson measure. We show that  first taking the large genus limit, then a short window limit, the distribution tends to a Gaussian. The variance was recently shown to be given by the corresponding quantity for the Gaussian Orthogonal Ensemble (GOE), and  the Gaussian fluctuations are also consistent with those in Random Matrix Theory,   as conjectured in the physics literature for a fixed surface.
\end{abstract}

\date{\today}
\maketitle


\section{Introduction}  
Let $X$ be a compact hyperbolic surface of genus $g\geq 2$, 
and $\lambda_j = 1/4+r_j^2$ be the eigenvalues of the Laplacian on $X$.  
We examine the statistics of these eigenvalues in short(ish) intervals. For this purpose, we use a smooth linear statistic as in \cite{RGOE}: for an even test function $f$ with compactly supported Fourier transform $\^f \in C_c^\infty(\R)$ 
and $\tau>0$  define 
\[
N_{f,L,\tau}(X):= \sum_{j\geq 0} f\left(L\left(r_j-\tau \right)\right) +  f\left(L\left(r_j + \tau \right)\right)  . 
\]
(When $\tau=0$ the definition changes to $\sum_j f(Lr_j)$.)

In \cite{RGOE}  it was shown that in the double limit, of first averaging over the moduli space $\Mg$ of hyperbolic surfaces of genus $g$ with respect to the Weil-Petersson measure, taking the large genus limit $g\to \infty$, 
and then taking the limit $L\to \infty$, the variance of $N_{f,L,\tau}$ is given by the GOE variance: 
denoting by $\EWp_g$ the average over $\Mg$, we have 
\begin{equation}\label{eq:RGOE}
\lim_{L\to \infty} \left( \lim_{g\to \infty}  
\EWp_g\left( \left|  N_{L,\tau}-\EWp_g( N_{L,\tau})\right|^2 \right) \right) 
= \SigGOE(f) 
\end{equation}
where $\SigGOE(f) = 2\intinf |x|\^f(x)^2dx$. This supports conjectures of Michael Berry \cite{Berry1985, Berry1986} on number variance for chaotic systems. 

In this note, we examine  the distribution of  $N_{L,\tau}(X)$. 
The distribution is conjectured to be Gaussian for any fixed hyperbolic surface in the large energy limit \cite{AS}, see also \cite{RudnickCLT} for the modular surface, as one can expect based on analogous behaviour  in random matrix theory \cite{DE}, and for the zeros of the Riemann zeta function \cite{Selberg}. Here, we take the surface $X$ to be random with respect to the Weil Petersson measure on the moduli space $\Mg$, and show that in the double limit the distribution tends to a normal distribution: 

\begin{thm}\label{thm:char function Nintro}
For all bounded continuous functions $h$, we have
\[
\lim_{L\to \infty}\lim_{g\to \infty} \EWp_g\left( h\left(\frac{N_{L,\tau} -\EWp_g(N_{L,\tau} )}{\sqrt{\SigGOE(f)}}\right) \right)  = \frac 1{\sqrt{2\pi}} \intinf h(t)e^{-t^2/2}dt .
\]
\end{thm}

 The opening step in  the proof of Theorem~\ref{thm:char function Nintro} is to express $N_{L,\tau}$ as a sum over closed geodesics via the Selberg Trace Formula. We then apply a recent result of Mirzakhani and Petri \cite{MP}, that the set of lengths of primitive closed geodesics, when thought of as a random point process on the moduli space  $\Mg$, converge as $g\to \infty$, to a Poisson point process with a certain intensity. We thus obtain, for each $L$, a certain random variable, the value of a functional $\mathcal H_{L,\tau}$ on the Poisson point process, which we then show has a Gaussian limit distribution as $L\to \infty$.  
 
 A noteworthy feature of our proof of Theorem~\ref{thm:char function Nintro} is that we avoid using the method of moments, and instead use various  features of the theory of point processes to pass directly to the limiting Poisson point process, where the computation 
 of the limit $L\to \infty$ is greatly streamlined.

There are other random models of hyperbolic surfaces, and various spectral statistics in these  models have been explored recently \cite{HM, Shen-Wu},   
for instance in the random cover model, the analogue of \cite{RGOE} for the smooth number variance  has recently been obtained 
 by Fr\'ed\'eric Naud  \cite{Naud}, who also obtained GUE statistics for the twisted Laplacian.   
The CLT in the random cover model will be in the MSc thesis of   Yotam Maoz \cite{Maoz}.

\section{Background on point processes}

\subsection{Generalities}
We give  brief background on point processes, using the survey \cite{Grandell} as our basic reference. 

Our background space will be  the non-negative reals  $\Rnneg= [0,\infty)$.  
A point process on $\Rnneg$ is a random assignment of a set of points in $\Rnneg$, each of which is assumed to be locally finite. 
We assume that the origin is not one of these points. We denote by $\Nc$ the set of all realizations of  point processes on $\Rnneg$. These can be thought of as atomic measures   
\begin{equation}\label{atomic rep}
\mu = \sum_{j=1}^\infty  \delta(x_j)
\end{equation}
with $0<x_1\leq x_2\leq \dots$ a discrete set of points (so the only possible accumulation point is at infinity) each occurring with finite multiplicity. In particular, for such $\mu$, we have  $\mu(0)=0$ and $\mu\{(0,t]\}<\infty$.  

We have a topology on $\Nc$, given by declaring that a sequence of measures $\mu_n\in \Nc$ converges vaguely to another measure $\mu_\infty\in \Nc$, 
if  for each continuity point of $ \mu_\infty\{(0,t]\}$, we have $\lim_{n\to \infty}\mu_n\{(0,t]\}=\mu\{(0,t]\}$.  We denote this by $\mu_n\vlim \mu$. 
 Equivalently \cite[Theorem 1]{Grandell}, for all {\em compactly supported} continuous functions $G$ on $\Rnneg$  we have
\begin{equation}\label{functional def vague conv}
\lim_{n\to \infty} \int_0^\infty G d\mu_n = \int_0^\infty G d\mu_\infty .
\end{equation}
With this topology, $\Nc$ becomes a separable, complete space which is metrizable \cite{Grandell}.

For a point process $\mathbf N$ on $\Rnneg$ and a bounded Borel set $B\subset \Rnneg$, we denote by $\mathbf N(B)$ the number of points in $B$; this is a random variable. The measure $\lambda(B) = \E(\mathbf N(B))$ is called the intensity  of the process. 

As an important example, a Poisson point process with intensity   $\lambda$ is a point process $\Pois(\lambda)$ so that for any Borel set $B$, the random variables $\mathbf N(B)$ is a Poisson variable with intensity  $\lambda(B)$, and so that for any choice of {\em disjoint} Borel sets $B_1,\dots, B_k$, the random variables $\mathbf N(B_1),\dots, \mathbf N(B_k)$ are {\em independent}.

Given point processes $\mathbf N_n, \mathbf N_\infty$ on $\Rnneg$, we say that $\mathbf N_n$ converges in distribution to $\mathbf N_\infty$ (written $\mathbf N_n \dlim \mathbf N_\infty$) if the sequence of random vectors $(\mathbf N_n(B_1),\dots, \mathbf N_n(B_k))$ converges in distribution to the random vector $(\mathbf N_\infty(B_1),\dots, \mathbf N_\infty(B_k))$ for all $k\geq 1$ and all choices of bounded Borel sets $B_i$ with boundaries satisfying $\mathbf N_\infty(\partial B_i)=0$ almost surely for all $i$. 
This is equivalent to requiring  that 
\[
\E(h(\mathbf N_n))\to \E(h(\mathbf N_\infty))
\]
  for all  bounded continuous functions $h:\mathcal N\to \R$ (continuous means that whenever we have a sequence $\nu_n\in \Nc$ which converges vaguely to $\nu_\infty$, we have $\lim_n h(\nu_n) = h(\nu_\infty)$).

We next recall the continuous mapping theorem (see e.g. \cite[Theorem 9.4.2]{AL}), which in our context,  states that if we have a sequence of random variables $X_n: \Omega_n\to \Nc$, each defined on its own probability space 
$\Omega_n$,  which converge in distribution to another random variable $X_\infty:\Omega_\infty\to \Nc$, and 
$\mathcal G:\Nc\to \R$ is a continuous map\footnote{We can also allow $\mathcal G$ to not be  continuous in a set of measure zero w.r.t. the distribution of $X_\infty$.}   then the random variables $\mathcal G(X_n):\Omega_n\to \R$ converge in distribution to the random variable $\mathcal G(X_\infty):\Omega_\infty\to \R$.  

\subsection{A functional}
 
It important for us to extend the functional form \eqref{functional def vague conv} of vague convergence to allow taking $G$ which is continuous on the positive reals, but allowed not to extend continuously to all of  $\Rnneg$, e.g. to blow up at $x=0$. 
We define a function   $\mathcal G:\Nc\to \R$ by
\[
\mathcal G:\mu\in \Nc  \mapsto  \int_0^\infty G d\mu := \sum_{x\in \mu} G(x).
\]

We claim that $\mathcal G:\Nc\to \R$ is  {\em continuous}: 
\begin{lem}\label{lem:extend vague conv}
Assume $G\in C(\Rpos)$ is continuous on the positive reals, and  supported in a bounded interval. 
Then $\mathcal G:\Nc\to \R$ is a   continuous mapping.  
\end{lem}

\begin{proof}
 What we need to show is that if $\mu_\infty,\mu_n\in \Nc$ with $\mu_n\vlim \mu_\infty$ vaguely, then $\lim_{n\to \infty} \mathcal G(\mu_n) = \mathcal G(\mu_\infty)$, that is \eqref{functional def vague conv}  holds for $G$. 

Denote the atoms of the measures $\mu_n$, $\mu_\infty$ in the representation \eqref{atomic rep} by $x_{n,j}$ ($n=1,2\dots, \infty$, $j=1,2,\dots$):
\[
\mu_n = \sum_{j\geq 1} \delta(x_{n,j}) .
\]
Note that vague convergence $\mu_n\to \mu_\infty $ implies convergence of the point sets, which are the discontinuity points of the distribution functions, cf  \cite[Theorem 25]{Grandell}: 
Suppose that $G$ is supported in $(0,L)$. If $x_{\infty,1}\dots, x_{\infty,J}$ are the atoms of $\mu_\infty$ in $(0,L)$, so that $\mathcal G(\mu_\infty) = \sum_{j=1}^J G(x_{\infty,j})$, 
then for all $n\gg 1$, there are exactly $J$ atoms $x_{n,1},\dots, x_{n,J}$ of $\mu_n$ in $(0,L)$ and  $\lim_{n\to \infty}x_{n,j}=x_{\infty,j}$ for all $j=1,\dots, J$. Hence for $n\gg 1$, 
\[
\lim_{n\to \infty} \mathcal G(\mu_n) =\lim_{n\to \infty}  \sum_{j=1}^J G(x_{n,j}) =  \sum_{j=1}^J \lim_{n\to \infty}G(x_{n,j})  .
 \]
Since $G(x)$ is continuous at $x_j$, we have 
\[
\lim_{n\to \infty}G(x_{n,j})  = G(\lim_n x_{n,j}) = G(x_{\infty,j}),
\]
 so that 
\[
\lim_{n\to \infty} \mathcal G(\mu_n) =\sum_{j=1}^J  G(x_{\infty,j})= \mathcal G(\mu_\infty)  
\]
as claimed. 
\end{proof}

\section{Reduction to a Poisson approximation}
 
 \subsection{An expansion}
 Using Selberg's trace formula, we expand  the centered variable $N_{L,\tau}-\EWp(N_{L,\tau})$ as a sum over closed geodesics  \cite{RGOE} 
\[
N_{L,\tau}-\EWp(N_{L,\tau}) = N^{osc} -\EWp(N^{osc})
\]
with
\begin{equation}\label{split N}
N^{osc}(X) =  \sum_{\gamma} H_{L,\tau}\left(\ell_\gamma\left(X\right)\right)
\end{equation}
the sum over all primitive,  non-oriented, closed geodesics $\{\gamma\}$ on $X$, with lengths $\ell_{\gamma}(X)$ where  
\[
H_{L,\tau}(x) =\frac{2x}{L}\sum_{k=1}^\infty F(kx), \quad F(x) =   \frac{\^ f\left( \frac { x}{L}\right) \cos( x \tau)}{ \sinh( x /2)} .
\]

For each $x>0$, the sum defining $H_{L,\tau}(x)$ is finite, ranging up to $k\leq L/x$, and so we get a continuous function on the positive reals. But   there need be no limit as $x\to 0$; if   $\^f(0)>0$  then for $x$ small, 
\[
H_{L,\tau}(x)  
\gg \sum_{k \ll x^{-1/2}}  \frac{x}{\sinh(kx)} \^f(0)
\gg \sum_{k \ll x^{-1/2}}  \frac{\^f(0)}{k} \gg \log \frac 1x 
\]  
which blows up as $x\to 0$. 

\subsection{The Mirzakhani-Petri Theorem}

 For each genus $g\geq 2$, the moduli space $\Mg$ equipped with the Weil-Petersson probability measure gives a probability space, and  the length spectrum gives a point process $\mathcal L_g:\Mg\to \Nc$, assigning to a surface $X\in \Mg$ its (primitive, unoriented)  length spectrum $\mathcal L_g(X)= \{0<\ell_1\leq \ell_2\leq \dots\}$.  For an interval $[a,b]\subset \Rnneg$, we have the random variable 
 \[
 N_g([a,b]):X\mapsto \# \mathcal L_g(X)\cap [a,b].
 \]  
 
 Further, let $\Pois(\nu_{MP})$ be the Poisson point process on $\Rnneg$  with intensity  
 \[
 \nu_{MP}(x)=\frac{2\sinh^2(x/2)}{x} dx,
 \]
  that is for each interval $I\subset \Rnneg$, the counting function $N_{MP}(I)$ is a Poisson random variable with intensity $\nu_{MP}(I) = \int_I \frac{2\sinh^2(x/2)}{x}dx$. 
 
 Mirzakhani and Petri \cite{MP} showed that for any set of disjoint intervals $I_i=[a_i,b_i]$, $i=1,\dots,k$,  the random variable $(N_g(I_1),\dots N_g(I_k))$ converge in distribution to the Poisson random variable $(N_{MP}(I_1),\dots N_{MP}(I_k))$, in other words that the point processes $\mathcal L_g$ converge in distribution to the Poisson point  process $\Pois(\nu_{MP})$.

\subsection{A   random  approximation}

Define a mapping $\Nc\to \R$ 
\[
\mathcal H_{L,\tau} = \sum_{\ell\in \mathcal L}  H_{L,\tau}(\ell)  .
 \]
  We note that $H_{L,\tau}$ is a continuous function on the positive reals. Therefore, by Lemma~\ref{lem:extend vague conv}, 
 the functional $\mathcal H_{L,\tau}:\Nc\to \R$ is {\em continuous} with respect to the vague topology on the space of point processes on the positive reals.

By the Mirzakhani-Petri theorem combined with the continuous mapping theorem, we find that the random variables $N^{osc}=\mathcal H_{L,\tau}\circ \mathcal L_g: \Mg\to \R$ converge in distribution to the random variable 
\[
S_{L,\tau}:=\mathcal H_{L,\tau}\circ \Pois(\nu_{MP}).
\] 

We recall that a sequence of real valued random variables $Y_i$ converges in distribution to a random variable $Y$ if the cumulative distribution functions  (CDF's) $F_i(t)=\Prob(Y_i\leq t)$ converge pointwise to $F_Y(t)$  in every continuity point of $F_Y$. This is equivalent to requiring that for every bounded continuous function $h$, we have 
\[
\lim_{i\to \infty} \E\left( h\left( Y_i \right) \right) = \E\left( h\left( Y\right) \right) .
\]
In turn this, by L\'evy's continuity theorem, is equivalent to pointwise convergence of the characteristic functions:
Recall  that for a random variable $Y$, the characteristic function is $\phi_Y(t) = \E(\exp(itY))$, $t$ real.  

Hence the characteristic functions  of $ N^{osc} $         converge as $g\to \infty$ to the characteristic function of $S_{L,\tau} $:
\begin{cor}
We have convergence of characteristic functions
\[
\lim_{g\to \infty} \EWp_g\left(\exp\left(it \left( N^{osc}  \right)\right)\right)  = 
\E_{\Pois}\left(\exp\left(it \left(S_{L,\tau} \right)\right)\right) .
\]
\end{cor}

\subsection{Working with the Poisson approximation}

 We will next show that  in the limit $L\to \infty$,  the characteristic function  of $S_{L,\tau}$ tends to that of a Gaussian. By L\'evy's continuity theorem (see e.g. \cite[Chapter 10]{AL}), this will prove Theorem~\ref{thm:char function Nintro}.

\begin{thm}\label{thm:char function N} 
For all   $\tau>0$,  
\[
\lim_{L\to \infty}\E_{\Pois}(it   \frac{ S_{L,\tau} -\E(S_{L,\tau} )}{\sqrt{\SigGOE(f)}}) =e^{-t^2/2}  .
\] 
\end{thm}

By its definition, the expected value is zero. That the variance is correct was demonstrated in the course of the proof of 
\cite[Lemma 5.2]{RGOE}. To show  Gaussianity, we use a combinatorial shortcut, of using cumulants.

Recall that the cumulants of a random variable $S$ are the coefficients in the Taylor expansion of the cumulant generating function $\log \E(e^{z S}) $ about $z=0$:
\[
\log \E(e^{z S})  = \sum_{m=1}^\infty \kappa_m(S) \frac{z^m}{m!}. 
\]
The first moment is the first cumulant, the second cumulant is just the variance, and the third cumulant is the centered third moment. In general, the $n$-th centered moment   is an $n$-th-degree polynomial in the first $n$ cumulants, for instance, the fourth centered moment is $\kappa_4+3\kappa_2^2$. 
Gaussianity   is equivalent to vanishing of the higher cumulants $\kappa_m(S)$ for $m\geq 3$. 
So we want to show:
\begin{prop}\label{Prop: cumulants of SLtau}
For all $m\geq 3$, 
\[
\lim_{L\to \infty} \kappa_m(S_{L,\tau}) = 0 . 
\]
\end{prop}

We will use Campbell's formula \cite{Kingman}: Let $\mathcal L$  be a Poisson point process with intensity  $\nu$,   let 
$H: \R\to \R$ be a measurable function,  and define the random sum 
\[
S = \sum_{\ell\in \mathcal L} H(\ell) .
\]
Then (assuming everything converges) the expected value and variance of $S$ are given by 
\[
\E(S) = \int_\R H(x)d\nu(x),\qquad \Var(S) = \int_\R H(x)^2 d\nu(x)
\]
and the moment generating function is given by 
\begin{equation}\label{moment gen f}
\E(e^{z S}) =\exp\left( \int_\R  [ e^{z H(x)}-1] d\nu(x) \right) .
\end{equation}

Using the formula \eqref{moment gen f} gives
\begin{equation*}
\log \E(e^{z S})   =
 \sum_{m\geq 1} \frac{z^m}{m!} \int_\R H(x)^m d\nu(x) .
\end{equation*}
Therefore, Proposition~\ref{Prop: cumulants of SLtau} follows from:
\begin{prop}
For all $m\geq 3$, uniformly in $\tau$, 
\begin{equation}\label{criterion gaussian m}
\lim_{L\to \infty}\int_\R H_{L,\tau}(x)^m d\nu_{MP}(x) = 0 .
\end{equation}
 \end{prop}

\begin{proof}
We use: 
\begin{lem}\cite[Lemma 6.3]{RGOE}\label{lem:bounds for HL}
Let $L>2$.  Then uniformly in $\tau$,

i) For $0<x<1/2$, we have $|H_{L,\tau}(x) | \ll \frac 1L \log (L/x)$.

ii) For $x\geq 1/2$, we have   $|H_{L,\tau}(x) | \ll \frac 1L  x \exp(-x/2)$.
\end{lem}

Applying Lemma~\ref{lem:bounds for HL} we obtain
\[
\begin{split}
\int_0^\infty H_{L,\tau}(x)^m d\nu_{MP}(x) &\ll \int_0^{1/2} \frac 1{L^m}(\log \frac Lx )^m \frac{\sinh(x/2)^2 }{x}dx 
\\
& \qquad +\int_{1/2}^L \frac 1{L^m} x^me^{-mx/2} \frac{\sinh(x/2)^2 }{x}dx .
\end{split}
\]
For $0<x<1/2$, use $\sinh(x/2)< 1.1 \cdot x/2\ll x$ to bound
\[
\begin{split}
\int_0^{1/2} (\log \frac Lx )^m \frac{\sinh(x/2)^2 }{x}dx &\ll \int_0^{1/2} x(\log \frac Lx )^m dx  
\\
&= L^2\int_{2L}^\infty (\log y)^m \frac{dy}{y^3} 
\\
&\ll_m  L^2\int_{2L}^{L^2} (\log L^2)^m \frac{dy}{y^3}  +L^2\int_{L^2}^\infty y^{1/2} \frac{dy}{y^3} 
\\ 
&\ll  (\log L)^m + L^{-3}\ll  \left(\log L \right)^m 
\end{split}
\]
so that the first integral contributes  $O( (\log L/L)^m )$. For the second integral, use $\sinh(x/2) <  e^{x/2}$ for $x>0$ to bound, for $m\geq 3$, 
\[
\int_{1/2}^L   x^me^{-mx/2} \frac{\sinh(x/2)^2 }{x}dx < \int_{0}^\infty e^{-(\frac m2-1)x} x^{m-1} dx =\frac{ 2^m \Gamma (m)}{ (m-2)^{m}}
\] 
so that the second integral contributes $O_m(L^{-m})$. Thus we obtain for $m\geq 3$ 
\[
\int_0^\infty H_{L,\tau}(x)^m d\nu_{MP}(x)  \ll_m L^{-m+o(1)} 
\]
which proves \eqref{criterion gaussian m}. 
\end{proof}

 This concludes the proof of Theorem~\ref{thm:char function N},  hence of Theorem~\ref{thm:char function Nintro}.

\end{document}